\documentclass[a4paper,12pt,leqno]{amsart}
\usepackage{amsmath,amsthm,amssymb,latexsym,epsfig,graphicx,subfigure}
\numberwithin{equation}{section}

\newtheorem{thm}{Theorem}[section]

\newtheorem{example}[thm]{Example}
\newtheorem{question}[thm]{Question}

\newtheorem{cor}[thm]{Corollary}
 
\newtheorem{pr}[thm]{Proposition}
\theoremstyle{definition}

\theoremstyle{definition}

\newcommand{\Sn}{S^{n-1}}
\newcommand{\Rn}{\mathbb{R}^{n}}

\newcommand{\R}{\mathbb{R}}

\newcommand{\N}{\mathbb{N}}

\newcommand {\grtrsim} {\ {\raise-.5ex\hbox{$\buildrel>\over\sim$}}\ }

\newcommand{\khii}{\text{\lower -.4ex\hbox{$\chi$}}}
\DeclareMathOperator{\spt}{spt}

\DeclareMathOperator{\Int}{Int}

\newcommand{\spa}{\operatorname{span}}

\begin{document}
\title [Intersections of projections and exceptional plane sections]{Hausdorff dimension, intersections of projections and exceptional plane sections}
\author{Pertti Mattila and Tuomas Orponen}

\thanks{Both authors were supported by the Academy of Finland, and TO in particular through the grant \emph{Restricted families of projections and connections to Kakeya type problems.}} \subjclass[2000]{Primary 28A75} \keywords{Hausdorff dimension, orthogonal projection, plane section}

\begin{abstract} This paper contains new results on two classical topics in fractal geometry: projections, and intersections with affine planes. To keep the notation of the abstract simple, we restrict the discussion to the planar cases of our theorems. 

Our first main result considers the orthogonal projections of two Borel sets $A,B \subset \R^{2}$ into one-dimensional subspaces. Under the assumptions $\dim A \leq 1 < \dim B$ and $\dim A + \dim B > 2$, we prove that the intersection of the projections $P_{L}(A)$ and $P_{L}(B)$ has dimension at least $\dim A - \epsilon$ for positively many lines $L$, and for any $\epsilon > 0$. This is quite sharp: given $s,t \in [0,2]$ with $s + t = 2$, we construct compact sets $A,B \subset \R^{2}$ with $\dim A = s$ and $\dim B = t$ such that almost all intersections $P_{L}(A) \cap P_{L}(B)$ are empty. In case both $\dim A > 1$ and $\dim B > 1$, we prove that the intersections $P_{L}(A) \cap P_{L}(B)$ have positive length for positively many $L$. 

If $A \subset \R^{2}$ is a Borel set with $0 < \mathcal{H}^{s}(A) < \infty$ for some $s > 1$, it is known that $A$ is 'visible' from almost all points $x \in \R^{2}$ in the sense that $A$ intersects a positive fraction of all lines passing through $x$. In fact, a result of Marstrand says that such non-empty intersections typically have dimension $s - 1$. Our second main result strengthens this by showing that the set of exceptional points $x \in \R^{2}$, for which Marstrand's assertion fails, has Hausdorff dimension at most one.

\end{abstract}

\maketitle

\section{Introduction} 

According to Marstrand's projection theorem, originating in \cite{Mar}, a Borel set $A \subset \R^{n}$ projects orthogonally onto a set of dimension $\min\{\dim A,m\}$ on almost all $m$-dimensional subspaces. Furthermore, if $\dim A > m$, then the projections have positive $m$-dimensional measure almost surely, and if $\dim A > 2m$, then almost all projections have non-empty interior. In this paper, we prove variants of these results for projections of two Borel sets $A,B \subset \R^{n}$. We denote the Grassmannian manifold of $m$-dimensional linear subspaces $V \subset \R^{n}$ by $G(n,m)$, and $P_{V} \colon \R^{n} \to V$ stands for the orthogonal projection onto $V$. For any pair of non-empty sets $A,B \subset \R^{n}$, it is clear that $P_{V}(A) \cap P_{V}(B) \neq \emptyset$ for certain $V \in G(n,m)$. We ask: under what conditions are there positively many such $V$ (with respect to the Haar measure $\gamma_{n,m}$ on $G(n,m)$)? And what can we say about the dimension, measure, or topology of the intersections $P_{V}(A) \cap P_{V}(B)$? Our first main result provides some answers:

\newpage

\begin{thm}\label{main1}
Let $A$ and $B$ be Borel subsets of $\Rn$.
\begin{itemize}
\item[(i)] If $\dim A > m$ and  $\dim B > m$, then 
$$\gamma_{n,m}\left(\{V\in G(n,m): \mathcal H^m(P_V(A)\cap P_V(B))>0\}\right)>0.$$
\item[(ii)] If $\dim A > 2m$ and  $\dim B > 2m$, then 
$$\gamma_{n,m}\left(\{V\in G(n,m): \Int(P_V(A)\cap P_V(B))\neq \emptyset\}\right)>0.$$
\item[(iii)] If $\dim A > m, \dim B \leq m$ and $\dim A + \dim B > 2m$, then for every $\epsilon > 0$,
$$\gamma_{n,m}\left(\{V\in G(n,m): \dim P_V(A)\cap P_V(B) > \dim B - \epsilon\}\right)>0.$$
\end{itemize}
\end{thm}

How sharp is Theorem \ref{main1}? In (i), the strict inequalities $\dim A > m$ and $\dim B > m$ are obviously necessary. The situation in (ii) is only clear when $m = 1$: using Besicovitch sets, one can easily find a set of full Lebesgue measure in the plane, all of whose projections on lines have empty interior, see \cite[Example 11.5]{Mat2}. For $m > 1$, such examples are not known, and the question remains open. As for (iii), the following example establishes the sharpness of the bounds in the plane, at least:
\begin{example}\label{mainEx} Let $0 \leq t \leq 1 \leq s \leq 2$, and $s+t = 2$. Then, there exist compact sets $A,B\subset\R^2$ with $\dim A = s$ and $\dim B = t$ such that $P_L(A)\cap P_L(B)=\emptyset$ for $\gamma_{2,1}$ almost all $L\in G(2,1)$. 
\end{example}

Our initial motivation to study the questions in Theorem \ref{main1} was, in fact, an application to the dimension theory of plane sections. Marstrand proved in \cite{Mar} that if $s > 1$, and $A \subset \R^{2}$ is an $\mathcal{H}^{s}$ measurable set with $0 < \mathcal{H}^{s}(A) < \infty$, then $\dim [A \cap (x + L)] = s - 1$ for $\mathcal{H}^{s} \times \gamma_{2,1}$ almost all point-line pairs $(x,L) \in A \times G(2,1)$. The higher dimensional generalisation is due to the first author \cite{Mat}: if $0 < \mathcal{H}^{s}(A) < \infty$ for $m < s < n$, then $\dim [A \cap (V + x)] = s - m$ for $\mathcal{H}^{s} \times \gamma_{n,n - m}$ almost all pairs $(x,V) \in A \times G(n,n - m)$. 

In the present paper, we are interested in exceptional set estimates for the results above. The reasonable question seems to be the following: for how many points $x \in \R^{n}$ can it happen that $\dim [A \cap (x + L)] < s - m$ for $\gamma_{n,n - m}$ almost all planes $V \in G(n,n - m)$? Applying (i) of Theorem \ref{main1}, we prove that such exceptional points $x$ are contained in a set of dimension at most $m$:
\begin{thm}\label{main2}
Let $m<s\leq n$ and let  $A\subset\Rn$ be $\mathcal H^s$ measurable with $0<\mathcal H^s(A)<\infty$. Then, the set $B \subset \R^{n}$ of points $x \in \R^{n}$ with
$$\gamma_{n,n-m}(\{V\in G(n,n-m):\dim A\cap(V+x)=s-m\}) =0 $$ 
has $\dim B \leq m$.
\end{thm}

An earlier result in this vein was obtained in $\R^{2}$ by the second author: the main result of \cite{O} states (after some trickery with projective transformations) that $\dim [B \cap L] \leq 2 - s$ for every line $L \subset \R^{2}$. For all we know, it is plausible that $\dim B \leq 2 - s$:
\begin{question}\label{sharpBound} Assume that $1 < s \leq 2$ and $A \subset \R^{2}$ is $\mathcal{H}^{s}$ measurable with $0 < \mathcal{H}^{s}(A) < \infty$. Is it true that the set $B \subset \R^{2}$ of points $x \in \R^{2}$ with
\begin{displaymath} \gamma_{2,1}(\{L \in G(2,1) : \dim [A \cap (L + x)] = s - 1\}) = 0 \end{displaymath}
has $\dim B \leq 2 - s$? \end{question} 
If true, this bound would be sharp: considering the sets $A,B$ in Example \ref{mainEx}, it is clear that $\gamma_{2,1}(\{L \in G(2,1) : A \cap (L + x) \neq \emptyset\}) = 0$ for all $x \in B$. In particular, for $t < 1 < s$, this means that $\gamma_{2,1}(\{L \in G(2,1) : A \cap (L + x) = s - 1\}) = 0$ for all $x \in B$, and hence the exceptional set $B$ identified in Theorem \ref{main2} can have dimension at least $t = 2 - s$ in the plane.

Finally, we observe that the case $m = n - 1$ of Theorem \ref{main2} immediately gives the following corollary for \emph{radial projections}:
\begin{cor}\label{mainCor} If $A\subset\Rn$ is a Borel set with dimension $\dim A > n-1$, then $A$ projects radially onto a set positive $(n-1)$-measure from all points of $\Rn$, except those in a set of dimension at most $n-1$. In short, the set of points from which $A$ is not "visible" is at most $(n - 1)$-dimensional. \end{cor}

The bound in Corollary \ref{mainCor} is stronger than the one attainable by the transversality method of Peres and Schlag; for comparison, the main result in \cite{PS} would imply that the exceptional set has dimension at most $2n - 1 - \dim A$. In analogy with Question \ref{sharpBound}, it seems plausible to conjecture that the sharp bound is $2(n - 1) - \dim A$. 

The next section contains some preliminaries, and the proof of Theorem \ref{main1} is contained in Section \ref{projectionSection}. Section \ref{exampleSection} contains the details of Example \ref{mainEx}, and Theorem \ref{main2} is proved in Section \ref{sectionSection}.

\section{Preliminaries} 

For $A\subset\Rn$ we denote by $\mathcal M(A)$ the set of Borel measures $\mu$ with $0<\mu(A)<\infty$ and with compact support $\spt\mu\subset A$. The $s$-energy of $\mu$ is
$$I_s(\mu)=\iint|x-y|^{-s}\,d\mu x\,d\mu y=c(n,s)\int_{\Rn}|\widehat{\mu}(x)|^2|x|^{s-n}\,dx.$$
The Fourier transform of $\mu$ is defined by $\widehat{\mu}(x)=\int e^{-2\pi ix\cdot y}\, d\mu y.$ 
For the second equality, see, for example, \cite{Mat2}, Theorem 3.10. If $0<u<n$, and $\mu, \nu\in\mathcal M(\Rn)$ are two measures with $\int_{\Rn}|\widehat{\mu}(x)\widehat{\nu}(x)||x|^{u-n}\,dx<\infty$, then their mutual $u$-energy is given by 
\begin{equation}\label{mutual}
I_u(\mu,\nu)=\iint|x-y|^{-u}\,d\mu x\,d\nu y=c(n,u)\int_{\Rn}\widehat{\mu}(x)\overline{\widehat{\nu}(x)}|x|^{u-n}\,dx.
\end{equation}
The latter formula is stated in \cite{Mat2}, Section 3.5, for functions, but it extends to measures by standard convolution approximation. Notice that if $u=(s+t)/2$, then we have by Schwartz's inequality
\begin{equation}\label{mutual3}
\int|\widehat{\mu}(x)\widehat{\nu}(x)||x|^{u-n}\,dx\leq 
\left(\int|\widehat{\mu}(x)|^2|x|^{s-n}\,dx\right)^{1/2}\left(\int|\widehat{\nu}(x)|^2|x|^{t-n}\,dx\right)^{1/2},\end{equation}
so
\begin{equation}\label{mutual4}
I_{(s+t)/2}(\mu,\nu)\lesssim\left(I_s(\mu)I_t(\nu)\right)^{1/2}.
\end{equation}
It follows that (\ref{mutual}) is valid provided $I_s(\mu)<\infty, I_t(\nu)<\infty$ and $s+t\geq 2u$.

By classical results of Frostman, if $A\subset\Rn$ is a Borel set with $\dim A > s>0$, then there is $\mu\in\mathcal M(A)$ with $I_s(\mu)<\infty$, cf. \cite{Mat2}, Theorem 2.8. 

Suppose that $\mu\in\mathcal M(\Rn)$ and $I_m(\mu)<\infty$. Then for $\gamma_{n,m}$ almost all $V\in G(n,m)$ the image $P_{V\sharp}\mu$ of $\mu$ under the projection $P_V$ is absolutely continuous with the Radon-Nikodym derivative, which we also denote by $P_{V\sharp}\mu$, in $L^2(V)$, cf. \cite{Mat1}, Theorem 9.7. Moreover, we have the disintegration formula
\begin{equation}\label{disint}
\int f\,d\mu = \int_V\int f\,d\mu_{V,a}\,d\mathcal H^ma
\end{equation}
for non-negative Borel functions $f$. Here $\mathcal H^m$ denotes the $m$-dimensional Hausdorff measure, which on an $m$-plane is just the Lebesgue measure. The sliced measures $\mu_{V,a}$ have supports in $(V^{\perp}+a)\cap \spt\mu$.  Equation (\ref{disint}) is a standard disintegration formula and is proven for example in \cite{Mat1}, (10.6). It is stated there for continuous functions, but it extends immediately.

When $B\subset V$ is a Borel set and $f$ is the characteristic function of $P_V^{-1}(B)$, (\ref{disint}) becomes
$$\mu(P_V^{-1}(B))= \int_V\mu_{V,a}(P_V^{-1}(B))\,d\mathcal H^ma = \int_B\mu_{V,a}(\Rn)\,d\mathcal H^ma.$$
On the other hand, by the definition of the image measure and the Radon-Nikodym derivative,
$$\mu(P_V^{-1}(B))= P_{V\sharp}\mu(B) = \int_B P_{V\sharp}\mu(a)\,d\mathcal H^ma.$$
Hence for $\mathcal H^m$ almost all $a\in V$,
\begin{equation}\label{disint1}
\mu_{V,a}(\Rn)=P_{V\sharp}\mu(a)
\end{equation}

In case $0<s<m$, we have $I_s(P_{V\sharp}\mu)<\infty$ for $\gamma_{n,m}$ almost all $V\in G(n,m)$ provided $I_s(\mu)<\infty$, see the proof of Theorem 9.3 in \cite{Mat1}.

We will make often use of the following formula, see for example \cite{Mat2}, (24.2):
\begin{equation}\label{be5}
\int_{G(n,m)}\int_{V}f(x)\,d\mathcal H^mxd\gamma_{n,m}V=c(n,m)\int_{\Rn}|x|^{m-n}f(x)\,dx.
\end{equation}
It is valid for Borel functions $f$ with $\int_{\Rn}|x|^{m-n}|f(x)|\,dx<\infty$. 
When $m=1$, this is just the formula for integration in polar coordinates.

\section{Intersections of projections}\label{projectionSection}

In this section we prove the main theorem for projections, Theorem \ref{main1}, and some variants of it. We recall the statement:

\begin{thm}\label{int-proj0}
Let $A$ and $B$ be Borel subsets of $\Rn$.
\begin{itemize}
\item[(i)] If $\dim A > m$ and  $\dim B > m$, then 
$$\gamma_{n,m}\left(\{V\in G(n,m): \mathcal H^m(P_V(A)\cap P_V(B))>0\}\right)>0.$$
\item[(ii)] If $\dim A > 2m$ and  $\dim B > 2m$, then 
$$\gamma_{n,m}\left(\{V\in G(n,m): \Int(P_V(A)\cap P_V(B))\neq \emptyset\}\right)>0.$$
\item[(iii)] If $\dim A > m, \dim B \leq m$ and $\dim A + \dim B > 2m$, then for every $\epsilon > 0$,
$$\gamma_{n,m}\left(\{V\in G(n,m): \dim (P_V(A)\cap P_V(B)) > \dim B - \epsilon\}\right)>0.$$
\end{itemize}
\end{thm}

Part (i) is an immediate corollary of the following:

\begin{thm}\label{int-proj}
Let $\mu, \nu\in\mathcal M(\Rn)$ with $I_m(\mu)<\infty$ and  $I_{m}(\nu)<\infty$. Then 
$$\iint_V P_{V\sharp}\mu(v) P_{V\sharp}\nu(v)\, d\mathcal H^mv\,d\gamma_{n,m}V>0$$
and
$$\gamma_{n,m}(\{V\in G(n,m): \mathcal H^m\left(P_V(\spt\mu)\cap P_V(\spt\nu))>0\}\right)>0.$$
\end{thm}

\begin{proof}

As stated above, for $\gamma_{n,m}$ almost all $V\in G(n,m)$ the images of $\mu$ and $\nu$  under the projection $P_V$ are absolutely continuous with the Radon-Nikodym derivatives $P_{V\sharp}\mu$ and $P_{V\sharp}\nu$ in $L^2(V)$. Thus $P_{V\sharp}\mu P_{V\sharp}\nu\in L^1(V)$ for $\gamma_{n,m}$ almost all $V\in G(n,m)$. To prove the theorem it suffices to show that $\int_V P_{V\sharp}\mu P_{V\sharp}\nu\, d\mathcal H^m>0$ for 
$V\in G(n,m)$ in a set of positive measure, since $P_{V\sharp}\mu(a)P_{V\sharp}\nu(a)>0$ implies $a\in P_V(\spt\mu)\cap P_V(\spt\nu)$.

We immediately see from the definition of the Fourier transform that $\widehat{P_{V\sharp}\mu}(v)=\widehat{\mu}(v)$ for $v\in V$. Thus by 
Plancherel's formula in the $m$-dimensional space $V$,
$$\int_V P_{V\sharp}\mu(v) P_{V\sharp}\nu(v)\, d\mathcal H^mv  =\int_V \widehat{P_{V\sharp}\mu}(v) \overline{\widehat{P_{V\sharp}\nu}(v)}\, d\mathcal H^mv = \int_V \widehat{\mu}(v) \overline{\widehat{\nu}(v)}\, d\mathcal H^mv.$$
Recalling  (\ref{mutual3}) we integrate over $G(n,m)$ and use equation (\ref{be5}) and (\ref{mutual}) to get
\begin{align}\label{mutual1}
\iint_V &P_{V\sharp}\mu(v) P_{V\sharp}\nu(v)\, d\mathcal H^mv\,d\gamma_{n,m}V  
= \iint_V \widehat{\mu}(v) \overline{\widehat{\nu}(v)}\, d\mathcal H^mv\,d\gamma_{n,m}V\\
&=c(n,m)\int_{\Rn} |x|^{m-n}\widehat{\mu}(x) \overline{\widehat{\nu}(x)}\, dx = c'(n,m)\iint|x-y|^{-m}\,d\mu x\,d\nu x>0.\notag
\end{align}
\end{proof}

Part (ii) of Theorem \ref{int-proj0} follows from the second part of following theorem:

\begin{thm}\label{int-proj2} Assume that $s + t = 2m$ and $\mu,\nu\in\mathcal M(\R^{n})$ with $I_{s}(\mu) < \infty$ and $I_{t}(\nu) < \infty$. Then 
\begin{equation}\label{form0} \gamma_{n,m}\left(\{V\in G(n,m): P_V(\spt\mu)\cap P_V(\spt\nu)\neq \emptyset\}\right)>0. \end{equation}
If also $s>2m$ and $t>2m$, then 
\begin{equation}\label{form1} \gamma_{n,m}\left(\{V\in G(n,m): \Int(P_V(\spt\mu)\cap P_V(\spt\nu))\neq \emptyset\}\right)>0. \end{equation}
\end{thm}

\begin{proof} Suppose that $s\geq t$ so that $P_{V\sharp}\mu$ is an $L^2$ function on $V$ for almost all $V$. Define $\mu_V(v)=P_{V\sharp}\mu(-v)$. 
By the above arguments we have again $\overline{\widehat{P_{V\sharp}\mu}}(v) \widehat{P_{V\sharp}\nu}\in L^1(V)$ with $\iint_V \overline{\widehat{P_{V\sharp}\mu}}(v) \widehat{P_{V\sharp}\nu}\, d\mathcal H^mv\,d\gamma_{n,m}V>0$. The inverse transform of $\overline{\widehat{P_{V\sharp}\mu}}(v) \widehat{P_{V\sharp}\nu}$ is $\mu_V\ast P_{V\sharp}\nu$. Thus for $\gamma_{n,m}$ positively many $V$, 
$\mu_V\ast P_{V\sharp}\nu$ is a continuous function on $V$ with $\mu_V\ast P_{V\sharp}\nu(0)>0$, that is
$$\int P_{V\sharp}\mu d P_{V\sharp}\nu = \mu_V\ast P_{V\sharp}\nu(0)>0.$$ 
This shows that the supports of $P_{V\sharp}\mu$ and $P_{V\sharp}\nu$ cannot be disjoint.

Suppose now that $s>2m$ and $t>2m$. Then by results of Falconer and O'Neil in \cite{FO} and Peres and Schlag in \cite{PS}, $\widehat{P_{V\sharp}\mu}, \widehat{P_{V\sharp}\nu}\in L^1(V)$ for $\gamma_{n,m}$ almost all $V\in G(n,m)$. This is very easy, let us check it for $\mu$: applying by (\ref{be5}) as in the previous proof and using Schwartz's inequality,
\begin{align*}
&\iint_{V\setminus B(0,1)} |\widehat{P_{V\sharp}\mu}(v)|\, d\mathcal H^mv\,d\gamma_{n,m}V  
=c(n,m)\int_{\Rn\setminus B(0,1)} |x|^{m-n}|\widehat{\mu}(x)|\,dx\\ 
&\leq c(n,m)\left(\int_{\Rn\setminus B(0,1)} |x|^{2m-s-n}|\,dx\right)^{1/2}\left(\int |x|^{s-n}|\widehat{\mu}(x)|^2\,dx\right)^{1/2}\lesssim I_s(\mu)^{1/2}<\infty.
\end{align*}
As $\widehat{P_{V\sharp}\mu}$ is bounded, the claim follows from this.

For any $V$ such that $\widehat{P_{V\sharp}\mu}, \widehat{P_{V\sharp}\nu}\in L^1(V)$, $P_{V\sharp}\mu$ and $\widehat{P_{V\sharp}\nu}$ are continuous. Above we found that 
$\int_V P_{V\sharp}\mu P_{V\sharp}\nu\, d\mathcal H^m >0$ for positively many $V$, whence also $\Int(P_V(\spt\mu)\cap P_V(\spt\nu))\neq \emptyset$ for positively many $V$.

\end{proof}

The above arguments also yield

\begin{thm}\label{int-proj1}
Suppose $2m<s<n$ and let $\mu\in\mathcal M(\Rn)$ with $I_{s}(\mu)<\infty$. Then 
$$\gamma_{n,m}\left(\{V\in G(n,m): 0\in \Int(P_V(\spt\mu))\}\right)>0.$$
\end{thm}

\begin{proof}
For almost all $V\in G(n,m)$ we again have that $\widehat{P_{V\sharp}\mu}\in L^1(V)$ and $P_{V\sharp}\mu$ is continuous. Then by the Fourier inversion formula
$$P_{V\sharp}\mu(0)=\int_V\widehat{P_{V\sharp}\mu}(v)\, d\mathcal H^mv.$$
Arguing as above,
\begin{align*}
&\int P_{V\sharp}\mu(0)\,d\gamma_{n,m}V=\iint_V\widehat{P_{V\sharp}\mu}(v)\, d\mathcal H^mv\,d\gamma_{n,m}V\\
&=\iint_V\widehat{\mu}(v)\, d\mathcal H^mv\,d\gamma_{n,m}V=c(n,m)\int_{\Rn} |x|^{m-n}\widehat{\mu}(x) \, dx.
\end{align*}
Letting $k_m$ be the Riesz kernel, $k_m(x)=|x|^{-m}$, the integrand $|x|^{m-n}\widehat{\mu}(x)$ is a constant multiple of the Fourier transform of $k_m\ast\mu$, see, for example (12.10) in \cite{Mat1}. Since, as above, $\int_{\Rn} |x|^{m-n}|\widehat{\mu}(x)| \, dx<\infty$ by Schwarz's inequality and the condition $I_s(\mu)<\infty$, we have again by the Fourier inversion formula,
$$\int P_{V\sharp}\mu(0)\,d\gamma_{n,m}V=c'(n,m)k_m\ast\mu(0)>0$$
As $P_{V\sharp}\mu$ is continuous for $\gamma_{n,m}$ almost all $V\in G(n,m)$, the theorem follows.
\end{proof}

\begin{proof}[Proof of Theorem \ref{int-proj0}(iii)]
Let $m<s<\dim A, 0<t<\dim B, s+t>2m$  and let $\mu\in\mathcal M(A), \nu\in\mathcal M(B)$ with $I_{s}(\mu)<\infty$ and $I_{t}(\nu)<\infty$. The proof of part (iii) of Theorem \ref{int-proj0} is again based on the identity (\ref{mutual1}). We shall apply it to a standard convolution approximation $\mu_{\delta}, \delta > 0,$ of $\mu$, in place of $\mu; \mu_{\delta}(x)=\psi_{\delta}\ast\mu(x), \psi_{\delta}(x)=\delta^{-n}\psi(x/\delta)$ where $\psi$ is a smooth non-negative function with support in $B(0,1)$ and with $\int\psi=1$. Then, using also (\ref{mutual3}), (\ref{mutual1}) takes the form
\begin{equation*}
\int\int_V P_{V\sharp}\mu_{\delta}(v)\, d P_{V\sharp}\nu v\,d\gamma_{n,m}V   = c(n,m)\int_{\Rn} |x|^{m-n}\widehat{\mu_{\delta}}(x) \overline{\widehat{\nu}(x)}\, dx.
\end{equation*} 
Since $\widehat{\mu_{\delta}}(x)=\widehat{\psi_{\delta}}(x)\widehat{\mu}(x)$ and $\widehat{\psi_{\delta}}(x)=\widehat{\psi}(\delta x)\to\widehat{\psi}(0)=1$ as $\delta\to 0$, we see that the right hand side tends to the positive and finite number $c(n,m)\int_{\Rn} |x|^{m-n}\widehat{\mu}(x) \overline{\widehat{\nu}(x)}\, dx=c'(n,m)I_{m}(\mu,\nu)$.  Hence there are $0<c<C<\infty$ such that for all $0<\delta<1$, 
\begin{equation}\label{mutual2}
c<\int\int_V P_{V\sharp}\mu_{\delta}(v)\, d P_{V\sharp}\nu v\,d\gamma_{n,m}V < C.
\end{equation} 

We have $P_{V\sharp}\mu_{\delta}=\psi^V_{\delta}\ast P_{V\sharp}\mu$ where $\psi^V(v)=\int_{V^{\perp}}\psi(v+w)d\mathcal H^{n-m}w$ for $v\in V$. Using again the identity $\widehat{P_{V\sharp}\mu}(v)=\widehat{\mu}(v)$ for $v\in V$ and the formula (\ref{be5}), we find that 
$$\int_{G(n,m)} \int_V|v|^{s-m}|\widehat{P_{V\sharp}\mu}(v)|^2\,d \mathcal H^{m}vd\gamma_{n,m}V=c(n,m)\int|x|^{s-n}|\widehat{\mu}(x)|^2\,d x<\infty.$$
This yields that for $\gamma_{n,m}$ almost all $V\in G(n,m)$, $P_{V\sharp}\mu$ is a function in the fractional Sobolev space $H^{(s-m)/2}(V)$, see, for example, \cite{Mat2}, Section 17.1. Consider the maximal function
$$M_Vf(v)=\sup_{\delta>0}|\psi^V_{\delta}\ast f(v)|, \qquad v\in V.$$ 
Since $m-2((s-m)/2)=2m-s<t$ and $I_t(P_{V\sharp}\nu)<\infty$ for $\gamma_{n,m}$ almost all $V\in G(n,m)$, we conclude for such $V$ from Theorem 17.3 in \cite{Mat2} (or rather its proof) that $\psi^V_{\delta}\ast P_{V\sharp}\mu$ converges $P_{V\sharp}\nu$ almost everywhere to $P_{V\sharp}\mu$ and $M_V(P_{V\sharp}\mu)\in L^1(P_{V\sharp}\nu)$. Hence by the dominated convergence theorem, $\psi^V_{\delta}\ast P_{V\sharp}\mu$ converges to $f_V:=P_{V\sharp}\mu|\spt P_{V\sharp}\nu$ in $L^1(P_{V\sharp}\nu)$.

Then by (\ref{mutual2}) for $\gamma_{n,m}$ positively many $V, f_V$ is positive and finite in a set of positive $P_{V\sharp}\nu$ measure. Thus with a large enough constant $C_V$ the measure $\chi_{\{x:f_V(x)\leq C_V\}}f_VP_{V\sharp}\nu$ is a nontrivial measure with finite $t$-energy and with support contained in $P_V(\spt\mu)\cap P_V(\spt\nu)$. Part (iii) of Theorem \ref{int-proj0} follows from this.
\end{proof}

\section{Sharpness of the projection theorem}\label{exampleSection}

This section contains the details of Example \ref{mainEx}. Recall the statement:

\begin{example}\label{int-proj3}
Let $s$ and $t$ be positive numbers such that $s+t = 2$ and $0 < t < 1 < s < 2$. Then, there exist compact sets $A,B\subset\R^2$ with $\dim A = s$ and $\dim B = t$ such that $P_L(A)\cap P_L(B)=\emptyset$ for $\gamma_{2,1}$ almost all $L\in G(2,1)$. 
\end{example}

\begin{proof} For a start, let us assume that we have located compact sets $A_{1},B \subset \R$ and $A_{2} \subset (0,\infty)$ such that $\dim (A_{1} \times A_{2}) = s$ and $\dim B = t$, and $\mathcal{H}^{1}(A_{1} + BA_{2}) = 0$. Of course, we will eventually indicate how such sets can be constructed, but before that, we explain how their existence implies the desired example. 

Write $C := A_{1} + BA_{2}$, and let $\pi_{b}$, $b \in B$, be the projection $\pi_{b}(x,y) := (x,y) \cdot (1,b) = x + by$. Then $\pi_{b}(A_{1} \times A_{2}) \subset C$ for all $b \in B$. Assume that $(x,y) \in A_{1} \times A_{2}$, $\pi_{b}(x,y) = c \in C$, and write $(x,y) = t(-b,1) + (c',0)$ for some $t,c' \in \R$. Then
\begin{displaymath} c = x + by = (-tb + c') + bt = c'. \end{displaymath}
So, the assumption $\pi_{b}(A_{1} \times A_{2}) \subset C$, $b \in B$, can be rewritten in the following way: 
\begin{equation}\label{form10} A_{1} \times A_{2} \subset \bigcap_{b \in B} \bigcup_{c \in C} \ell(-b,c), \end{equation} 
where $\ell(b,c)$ is the line $\ell(b,c) = \spa(b,1) + (c,0)$. Since $A_{2} \subset (0,\infty)$, the inclusion continues to hold, if we replace $\ell(-b,0)$ by $\ell(-b,0) \setminus \{(c,0)\}$. 

To make further progress, we employ the projective transformation
\begin{displaymath} F(x,y) := \frac{(x,1)}{y}, \qquad y \neq 0. \end{displaymath} 
The mapping $F$ has the useful property that it maps any line $\ell'(b,e) := (b,0) + \spa(e)$, with $b \in \R$ and $e = (e_{1},e_{2}) \in S^{1}$, $e_{2} \neq 0$, to the line $\ell(b,e_{1}/e_{2})$, as defined above. To be precise,
\begin{displaymath} F(\ell'(b,e) \setminus \{(b,0)\}) = \ell(b,e_{1}/e_{2}) \setminus \{(e_{1}/e_{2},0)\}. \end{displaymath} Indeed, if $e_{2} \neq 0$, any point $x \in \ell'(b,e) \setminus \{(b,0)\}$ can be written as 
\begin{displaymath} x = (b,0) + \frac{e}{te_{2}} = \left(b + \frac{e_{1}}{te_{2}},\frac{1}{t} \right), \quad t \in \R \setminus \{0\}, \end{displaymath}
and then
\begin{displaymath} F(x) = \frac{(b + e_{1}/(te_{2}),1)}{1/t} = (tb + e_{1}/e_{2},t) = t(b,1) + (e_{1}/e_{2},0) \in \ell(b,e_{1}/e_{2}). \end{displaymath}
It is also clear that any point in $\ell(b,e_{1}/e_{2}) \setminus \{(e_{1}/e_{2},0)\}$ can be obtained in this way.

 So, from \eqref{form10} and the sentence right under it, we infer that
\begin{displaymath} A := F^{-1}(A_{1} \times A_{2}) \subset \bigcap_{b \in B} \bigcup_{e \in E} \ell'(-b,e) \setminus \{(-b,0)\} = \bigcap_{b \in -B} \bigcup_{e \in E} \ell'(b,e) \setminus (-B \times \{0\}), \end{displaymath}
where $E \subset S^{1}$ is the set of those vectors $e = (e_{1},e_{2})$ such that $e_{1}/e_{2} \in C$. Note that $A$ is compact with $\dim A = \dim (A_{1} \times A_{1}) = s$ (rather: $A$ can be made compact by ensuring that $A_{1} \times A_{2}$ does not lie close to the singularities of $F^{-1}$, which is no problem later on). Also, since $C$ is a null-set, we have $\mathcal{H}^{1}(E) = 0$. Now, we claim that $P_{L}(-B \times \{0\}) \cap P_{L}(A) = \emptyset$, whenever $L \in G(2,1) \setminus \{e^{\perp} : e \in E\}$. This proves the proposition, since $\gamma_{2,1}(\{e^{\perp} : e \in E\}) = 0$. To prove the claim, assume that $P_{L}(A) \cap P_{L}(-B \times \{0\}) \neq \emptyset$. Thus, 
\begin{displaymath} P_{L}(a - (b,0)) = P_{L}(a) - P_{L}(b,0) = 0 \end{displaymath}
for some $a \in A$ and $(b,0) \in -B \times \{0\}$ with $a \neq (b,0)$, which is another way of saying that the difference $a - (b,0)$ lies in the orthogonal complement of $L$. But, since $a \in \ell'(b,e)$ for some $e \in E$, and the difference $a - (b,0)$ is parallel to $e$, this means that $e \in L^{\perp}$, and hence $L = e^{\perp}$. 

We have now reduced matters to establishing the claim formulated in the first paragraph of the proof, namely finding $A_{1},A_{2},B$ such that $\mathcal{H}^{1}(A_{1} + BA_{2}) = 0$. We will indicate the idea and leave the standard details for the reader. Let $r := 1/s \in (1/2,1)$, and define 
\begin{displaymath} A'_{1} := A_{1}'(n) := \left\{\frac{k}{n^{r}} : 1 \leq k \leq n^{r} \right\}, \quad A_{2}' := A_{2}'(n) := \left\{\frac{k}{n^{1 - r}} : 1 \leq k \leq n^{1 - r} \right\}, \end{displaymath}
and
\begin{displaymath} B' := B'(n) := \left\{\frac{k}{n^{2r - 1}} : 1 \leq k \leq n^{2r - 1}\right\}. \end{displaymath}
The main observation is that
\begin{displaymath} A_{2}'B' = \left\{ \frac{jk}{n^{r}} : 1 \leq j \leq n^{1 - r} \text{ and } 1 \leq k \leq n^{2r - 1} \right\} \subset \left\{ \frac{k}{n^{r}} : 1 \leq k \leq n^{r} \right\},  \end{displaymath}
which implies that
\begin{displaymath} A_{1}' + A_{2}'B' \subset \left\{\frac{k}{n^{r}} : 2 \leq k \leq 2n^{r}\right\}. \end{displaymath}
Consequently, if $r < r'$, and $A_{i} = A_{i}(n,r')$ and $B = B(n,r')$ are defined as the $n^{-r'}$-neighbourhoods of $A_{i}'$ and $B'$, respectively, then $A_{1} + A_{2}B$ can be covered by intervals of length $\sim n^{-r'}$ centred at the $\leq 2n^{r}$ points in $A_{1}' + A_{2}'B$. In particular, $\mathcal{H}^{1}(A_{1} + A_{2}B) \lesssim n^{r - r'} \to 0$, as $n \to \infty$. We next calculate the "dimensions" of the sets $A_{1}$,$A_{2}$ and $B$ (this is a bit vague on purpose). Since $A_{1}$ consists of $n^{r}$ well-separated intervals of length $n^{-r'}$, the "dimension" of $A_{1}$ can be chosen arbitrarily close to one by choosing $r'$ close to $r$. Since $A_{2}$ consists of $n^{1 - r}$ well-separated intervals of length $n^{-r'}$ (here one needs that $n^{1 - r} < n^{r}$), the "dimension" of $A_{2}$ is $(1 - r)/r'$, which can be made arbitrarily close to $1/r - 1 = s - 1$ by choosing $r'$ close to $r$. Hence, the "dimension" of $A_{1} \times A_{2}$ can be made arbitrarily close to $1 + (s - 1) = s$, as required. Along the same lines, the "dimension" of $B$ can be made arbitrarily close to $2 - 1/r = 2 - s = t$.  

To make everything precise, one needs to perform a standard Cantor set construction, and the sets $A_{1}(n,r'),A_{2}(n,r')$ and $B(n,r')$ are only the first stage. One chooses a rapidly increasing sequence of integers $(n_{j})_{j \in \N}$, and a sequence positive reals $(r_{j}')_{j \in \N}$ with the properties that $r_{j}' > r$, and $r_{j} \searrow r$ as $j \to \infty$. Then, one repeatedly places a scaled copy of $A_{i}(n_{j + 1},r_{j + 1}')$ inside each constituent interval of $A_{i}(n_{j},r_{j}')$. If $A_{i}^{j}$ is the union of intervals obtained after $j$ iterations, one can arrange $\mathcal{H}^{1}(A_{1}^{j} + B(n_{j},r_{j})A_{2}^{j}) \leq 1/j$ by choosing the growth speed of the sequence $(n_{j})$ great enough. The sets $A_{1},A_{2}$ are the limit sets of this procedure, and then $\dim (A_{1} \times A_{2}) = s$, since $r_{j}' \to r$. Finally, by choosing the growth speed of the sequence $(n_{j})$ great enough, one can also ascertain that $\dim B = t$ with
\begin{displaymath} B := \bigcap_{j = 1}^{\infty} B(n_{j},r_{j}). \end{displaymath} 
Then $\mathcal{H}^{1}(A_{1} + BA_{2}) \leq \mathcal{H}^{1}(A_{1}^{j} + B(n_{j})A_{2}^{j}) \leq 1/j$ for all $j \in \N$, hence $\mathcal{H}^{1}(A_{1} + BA_{2}) = 0$. We leave a more detailed proof to the reader. To ensure that $A_{2} \subset (0,\infty)$, as required in the first part of the proof, one can replace $A_{2}$ by $A_{2} \cap [1/10,\infty)$ in the very end; this has no effect on the dimension.  \end{proof}

\section{Plane sections}\label{sectionSection}

We start with a proposition, which follows from Lemma 6.5 in \cite{Mat}, but we give here a slightly different simple proof.

\begin{pr}\label{sect-pr}
Let $A\subset\Rn$ be a Borel set with $\dim A>m$. Then for all $x\in\Rn$, $\dim A\cap(V+x)\leq\dim A-m$ for $\gamma_{n,n-m}$ almost all 
$V\in G(n,n-m)$. 
\end{pr}

\begin{proof}
Translating $A$, we may assume that $x=0$. Suppose first that $n-m=1$. Let $\pi_0$ be the radial projection from the origin onto the unit sphere $\Sn$. Then $\pi_0$ is locally Lipschitz with Lipschitz constant $1/\delta$ in $\Rn\setminus B(0,\delta)$ for $\delta>0$. Let $t>\dim A$. Then $\mathcal H^t(A)=0$ and by Theorem 7.7 in \cite{Mat1},
$$\int_{\Sn}\mathcal H^{t-(n-1)}((A\setminus B(0,\delta)\cap\pi_0^{-1}\{y\})\,d\mathcal H^{n-1}y\leq C(m,t,\delta)\mathcal H^t(A)=0.$$
Applying this to a sequnce $\delta_j\to 0$, it follows that $\mathcal H^{t-(n-1)}(A\cap\pi_0^{-1}\{y\})=0$ for $\mathcal H^{n-1}$ almost all $y\in\Sn$. As $t>\dim A$ was arbitrary, we have $\dim A\cap\pi_0^{-1}\{y\}\leq \dim A - (n-1)$ for $\mathcal H^{n-1}$ almost all $y\in\Sn$. But this is exactly the desired statement, since the punctured lines $L\setminus\{0\}$ are of the form $\pi_0^{-1}\{y\}\cup\pi_0^{-1}\{-y\}, y\in\Sn$ with $L\cap\Sn=\{y, -y\}$, and the surface measure $\mathcal H^{n-1}$ on $\Sn$ and the measure $\gamma_{n,1}$ are related by
$$\gamma_{n,1}(G)=c(n)\mathcal H^{n-1}(\bigcup_{L\in G}L\cap\Sn).$$

Suppose then $n-m>1$. For $W\in G(n,m+1)$, let
$$G(W,1)=\{L\in G(n,1):L\subset W\},$$
and let $\gamma_{W,1}$ be the natural measure on $G(W,1)$. We can write the measure $\gamma_{n,n-m}$ as
\begin{equation}\label{gamma}
\gamma_{n,n-m}(G)=\int\gamma_{W,1}(\{L\in G(W,1):W^{\perp}+L\in G\})\,d\gamma_{n,m+1}W
\end{equation}
for Borel sets $G\subset G(n,n-m)$. This holds since the right hand side defines an orthogonally invariant Borel probability measure on $G(n,n-m)$ and such a measure is unique.

Fix $W\in G(n,m+1)$ for a while. Every $x\in\Rn$ can be written uniquely as $x=v+w, v\in W, w\in W^{\perp}$. The map $\pi_W:x\mapsto v/|v|$ from $\Rn\setminus W^{\perp}$ onto $W\cap\Sn$ is  locally Lipschitz with $\pi_W^{-1}\{y\}\cup\pi_W^{-1}\{-y\}=W^{\perp}+L$ when $L\in G(W,1)$ with $L\setminus\{0\}=\pi_W^{-1}\{y\}\cup\pi_W^{-1}\{-y\}\cap W$. The same argument as in the case $m=n-1$ gives that $\dim A\cap(W^{\perp}+L)\leq\dim A-m$ for $\gamma_{W,1}$ almost all $L\in G(W,1)$. Now the proposition follows using (\ref{gamma}).
\end{proof}

\begin{thm}\label{sect-thm}
Let $m<s\leq n$ and let  $A\subset\Rn$ be $\mathcal H^s$ measurable with $0<\mathcal H^s(A)<\infty$. Then there is a Borel set $B\subset\Rn$ with $\dim B\leq m$ and with the following property: for every $x\in\Rn\setminus B$,
$$\gamma_{n,n-m}(\{V\in G(n,n-m):\dim A\cap(V+x)=s-m\})>0.$$ 
\end{thm}

\begin{proof}
Due to Proposition \ref{sect-pr}, we only need to prove the lower bound $\dim A\cap(V+x)\geq s-m$. We may assume that $A$ is compact and $\mathcal H^s(A)<\infty$ since, by the Borel regularity of Hausdorff measures, $A$ contains a compact set with positive and finite measure. Then the function $(x,V)\mapsto \dim A\cap(V+x), x\in\Rn, V\in G(n,n-m),$ is a Borel function. This is rather easy to see, or one can consult \cite{MM}. Denoting by $\mu$ the restriction of $\mathcal H^s$ to $A$ we may also assume that $I_m(\mu)<\infty$; for this it suffices that $\mu(B(x,r))\leq Cr^s$ for all balls $B(x,r)$ and this is achieved applying the upper density estimate Theorem 6.2(1) in \cite{Mat1} and restricting $\mu$ to a further subset.

Suppose that the assertion of the theorem fails. Then there is a Borel set $B\subset\Rn$ such that  $\dim B > m$ and for $x\in B$,~ $\dim A\cap(V+x)<s-m$ for $\gamma_{n,n-m}$ almost all $V\in G(n,n-m)$. Then we can find $\nu\in\mathcal M(B)$  such that $I_m(\nu)<\infty$. Now we have by Theorem 10.10 in \cite{Mat1} for $\mu$ almost all $x\in\Rn$, $\dim A\cap(V+x)\geq s-m$, and by the definition of $B$ for $\nu$ almost all $y\in\Rn$, $\dim A\cap(V+y)<s-m$, both for $\gamma_{n,n-m}$ almost all $V\in G(n,n-m)$. By Fubini's theorem, for $\gamma_{n,n-m}$ almost all $V\in G(n,n-m),~ \dim A\cap(V+x)\geq s-m$ for $\mu$ almost all $x\in\Rn$ and $\dim A\cap(V+y)<s-m$ for $\nu$ almost all $y\in\Rn$. We get a contradiction if we find such $x$ and $y$ for which $V+x=V+y$, that is, $P_{V^{\perp}}x=P_{V^{\perp}}y$.

Let $V\in G(n,n-m)$ be such that both $P_{V^{\perp}\sharp}\mu$ and $P_{V^{\perp}\sharp}\nu$ are absolutely continuous with respect to $\mathcal H^{m}$, which by \cite{Mat1}, Theorem 9.7, is true for almost all $V$. Define the Borel sets 
$$A_V=\{x\in\Rn: \dim A\cap(V+x)\geq s-m\},$$
$$B_V=\{y\in\Rn: \dim A\cap(V+y)< s-m\},$$
$$C_V=\{a\in V^{\perp}: P_{V^{\perp}\sharp}\mu(a)P_{V^{\perp}\sharp}\nu(a)>0\}.$$
Using the above properties of $\mu$ and $\nu$, we can find such a $V$ such that $A_V$ has full $\mu$ measure and $B_V$ has full $\nu$ measure, and moreover due to Theorem \ref{int-proj}, $\mathcal H^m(C_V)>0$. By (\ref{disint}),
$$\int_{V^{\perp}}\mu_{V^{\perp},a}(\Rn\setminus A_V)\,d\mathcal H^ma=\mu(\Rn\setminus A_V)=0,$$
and
$$\int_V\nu_{V^{\perp},a}(\Rn\setminus B_V)\,d\mathcal H^ma=\nu(\Rn\setminus B_V)=0.$$
Since by (\ref{disint1}) $\mu_{V^{\perp},a}(\Rn)=P_{V^{\perp}\sharp}\mu(a)$ and $\nu_{V^{\perp},a}(\Rn)=P_{V^{\perp}\sharp}\nu(a)$, it follows that 
\begin{displaymath} \mu_{V^{\perp},a}(A_V)>0 \quad \text{and} \quad \nu_{V^{\perp},a}(B_V)>0 \end{displaymath}
for $\mathcal H^m$ almost all $a\in C_V$. Hence there is a common $a\in C_V$ for which these measures are positive, yielding $x\in A_V$ and $y\in B_V$ with $P_{V^{\perp}\sharp}x=P_{V^{\perp}}y=a$ as desired.

\end{proof}

Using the existence of subsets of positive and finite Hausdorff measure, see, for example, \cite{Mat1}, Theorem 8.13, we get from Theorem \ref{sect-thm}

\begin{cor}
Let  $A\subset\Rn$ be a Borel set with $\dim A>s>m$. Then there is a Borel set $B\subset\Rn$ with $\dim B\leq m$ and with the following property: for every $x\in\Rn\setminus B$,
$$\gamma_{n,n-m}(\{V\in G(n,n-m):\dim A\cap(V+x)\geq s-m\})>0.$$ 
\end{cor}

This is false with $s=\dim A$: consider a countable union of compact sets $C_j, j=1,2,\dots,$ with $\dim C_j=s-1/j$ and diam$(C_j)$ tending to $0$.






\vspace{1cm}
\begin{footnotesize}
{\sc Department of Mathematics and Statistics,
P.O. Box 68,  FI-00014 University of Helsinki, Finland,}\\
\emph{E-mail addresses:} 
\verb"pertti.mattila@helsinki.fi", 
\verb"tuomas.orponen@helsinki.fi"

\end{footnotesize}

\end{document}